\documentclass[12pt]{article}
\usepackage{epsfig,amsmath,amsthm,amssymb,graphics,amsfonts,psfrag}
\usepackage[letterpaper,margin=1in]{geometry}

\theoremstyle{plain}
\newtheorem{thm}{Theorem}[section]

\newtheorem{prop}[thm]{Proposition}
\newtheorem{lem}[thm]{Lemma}

\newtheorem{cjt}[thm]{Conjecture}

\theoremstyle{definition}

\def\txtd{{\textnormal{d}}}
\def\txte{{\textnormal{e}}}
\def\txti{{\textnormal{i}}}

\def\txtD{{\textnormal{D}}}

\newtheoremstyle{myremark}% name
  {5pt}%      Space above
  {5pt}%      Space below
  {\sffamily}%         Body font
  {10pt}%         Indent amount (empty = no indent, \parindent = para indent)
  {\sffamily}% Thm head font
  {:}%        Punctuation after thm head
  {.5em}%     Space after thm head: " " = normal interword space;
  {}%         Thm head spec (can be left empty, meaning `normal')

\theoremstyle{myremark}
\newtheorem*{remark}{Remark}

\def\R{\mathbb{R}}
\def\C{\mathbb{C}}
\def\N{\mathbb{N}}

\newcommand{\be}{\begin{equation}}
\newcommand{\ee}{\end{equation}}
\newcommand{\bea}{\begin{eqnarray}}
\newcommand{\eea}{\end{eqnarray}}
\newcommand{\beann}{\begin{eqnarray*}}
\newcommand{\eeann}{\end{eqnarray*}}
\newcommand{\benn}{\begin{equation*}}
\newcommand{\eenn}{\end{equation*}}

\def\ra{\rightarrow}
\def\I{\infty}
\def\I{\infty}

\newcommand{\cB}{{\mathcal B}}  % calligraphic B
\newcommand{\cK}{{\mathcal K}}  % calligraphic K
\newcommand{\cL}{{\mathcal L}}  % calligraphic L
\newcommand{\cM}{{\mathcal M}}  % calligraphic M
\newcommand{\cN}{{\mathcal N}}  % calligraphic N
\newcommand{\cO}{{\mathcal O}}  % calligraphic O
\begin{document}

\author{Franz Achleitner and Christian Kuehn\footnotemark[1]}

\renewcommand{\thefootnote}{\fnsymbol{footnote}}
\footnotetext[1]{%
Institute for Analysis and Scientific Computing, 
Vienna University of Technology, 
1040 Vienna, Austria.} 
\renewcommand{\thefootnote}{\arabic{footnote}}
 
\title{On Bounded Positive Stationary Solutions\\ for a Nonlocal Fisher-KPP Equation}

\maketitle

\begin{abstract}
We study the existence of stationary solutions for a nonlocal version of the 
Fisher-Kolmogorov-Petrovskii-Piscounov (FKPP) equation. The main motivation is a
recent study by Berestycki et {al.} [Nonlinearity 22 (2009), {pp.}~2813--2844] where
the nonlocal FKPP equation has been studied and it was shown for the spatial domain $\R$ and
sufficiently small nonlocality that there are only two bounded non-negative stationary 
solutions. Here we provide a similar result for $\R^d$ using a completely different approach.
In particular, an abstract perturbation argument is used in suitable weighted
Sobolev spaces. One aim of the alternative strategy is that it can eventually be generalized 
to obtain persistence results for hyperbolic invariant sets for other nonlocal evolution 
equations on unbounded domains with small nonlocality, {i.e.}, to improve our understanding
in applications when a small nonlocal influence alters the dynamics and when it does not.
\end{abstract}


\section{Introduction}  
\label{sec:intro}

The (local) Fisher-Kolmogorov-Petrovskii-Piscounov (FKPP) equation is given by
\be
\label{eq:Fisher-KPP_classical}
\frac{\partial u}{\partial t}=\Delta u+\mu u(1- u),\qquad x\in\R^d,~ u:\R^d\times [0,\I)\ra \R,~u=u(x,t),
\ee
where $\mu>0$ is a parameter and $\Delta=\sum_{i=1}^d\frac{\partial^2 }{\partial x_i^2}$ 
denotes the Laplacian. Originally the equation was studied by Fisher \cite{Fisher} 
and Kolmogorov-Petrovskii-Piscounov \cite{KolmogorovPetrovskiiPiscounov} for $d=1$ with a focus 
on traveling waves connecting the two homogeneous stationary (or steady) states $u\equiv 0$ and 
$u\equiv 1$. The FKPP equation has been studied extensively as a standard model for invasion waves in 
mathematical biology \cite{Okubo,Murray1} and for propagation into unstable states in physics
\cite{EbertvanSaarloos}, often with a focus on the wave speed
\cite{BerestyckiHamelNadirashvili,BerestyckiHamelNadirashvili1}.\medskip 

This paper is directly motivated by the results of Berestycki et {al.}~\cite{BerestyckiNadinPerthameRyzhik}  
who considered a nonlocal version of FKPP equation 
\be
\label{eq:Fisher-KPP}
\frac{\partial u}{\partial t}=\Delta u+\mu u(1-\phi\ast u),
\ee
where $\phi:\R^d\ra\R$ is a kernel for the convolution 
\benn
(\phi\ast u)(x):=\int_{\R^d}u(x-y)\phi(y)~\txtd y=\int_{\R^d}u(y)\phi(x-y)~\txtd y.
\eenn
The term $\phi \ast u$ models nonlocal saturation or competition effects \cite{LefeverLejeune}. 
The nonlocal FKPP equation \eqref{eq:Fisher-KPP} has been studied from various perspectives 
\cite{AlfaroCoville,Britton1,Britton,FangZhao,GenieysVolpertAuger,Gourley,NadinPerthameTang} 
with a focus on stability analysis of steady states, the existence of traveling waves and applications
in mathematical biology; there 
are also several results available for bounded domains \cite{FurterGrinfeld,PerthameGenieys} instead 
of $\R^d$. Other types of nonlocality may arise instead of the nonlinear term $u(\phi\ast u)$ 
\cite{CovilleDupaigne,CovilleDupaigne1,PerthameSouganidis,VougalterVolpert}, generalizations of \eqref{eq:Fisher-KPP}
have been considered \cite{EbertvanSaarloos,WangLiRuan} as well as the nonlocal bistable case \cite{AlfaroCovilleRaoul}. 
Furthermore, there are several other different nonlocal versions of the FKPP equation involving time delay 
\cite{AshwinBartuccelliBridgesGourley,Zou} or nonlocal diffusion via fractional operators 
\cite{del-Castillo-NegreteCarrerasLynch,FelmerYangari,MancinelliVergniVulpiani}.\medskip 

In this paper we focus on the analysis of stationary solutions for the FKPP equation 
\eqref{eq:Fisher-KPP} which satisfy 
\be
\label{eq:ss_Fisher-KPP}
0=\Delta u+\mu u(1-\phi\ast u),\qquad x\in\R^d,~u:\R^d\ra \R,~x\mapsto u(x).
\ee 
The main goal is to understand certain subclasses of classical solutions $u\in C^2_b(\R^d,\R)=:C^2_b$ 
which satisfy \eqref{eq:ss_Fisher-KPP} pointwise. However, we shall need weaker solution spaces to infer 
properties about the relevant classical solutions. It is assumed in \cite{BerestyckiNadinPerthameRyzhik} that 
the kernel satisfies 
\be
\label{eq:kernel_assume}
\phi\geq 0,\qquad \phi(0)>0,\qquad \nabla \phi\in C_b(\R^d),
\qquad \int_{\R^d} \phi(x)~\txtd x=1,\qquad \int_{\R}x^2 \phi(x)~\txtd x<\I.
\ee
Observe that $u\equiv 0$ is a solution of \eqref{eq:ss_Fisher-KPP} and since $\int_{\R^d}\phi ~\txtd x=1$ 
it follows that $u\equiv 1$ is also always a solution of \eqref{eq:ss_Fisher-KPP}. 
Note that a Gaussian probability density and a (suitably extended) exponential probability 
density satisfy the assumptions \eqref{eq:kernel_assume}. Considering $\tilde{x}:=\sigma x$, a 
function $\tilde{u}(\tilde{x}):=u(\tilde{x}/\sigma)=u(x)$ and kernel $\phi_\sigma(\tilde{x}):=
\frac{1}{\sigma^d} \phi(\tilde{x}/\sigma)=\frac{1}{\sigma^d}\phi(x)$ one finds, upon dropping the 
tildes and letting $\mu=\sigma^2$, that
\be
\label{eq:ss_Fisher-KPP_space_scaled}
0=\Delta u+u(1-\phi_\sigma\ast u)=:F(u,\sigma),
\ee 
Hence, upon a space rescaling, the equations \eqref{eq:ss_Fisher-KPP} and 
\eqref{eq:ss_Fisher-KPP_space_scaled} are equivalent if $\mu\neq0$ and $\sigma\neq 0$. 
We shall work with the version 
\eqref{eq:ss_Fisher-KPP_space_scaled} from now on. The kernel $\phi_\sigma$ converges to a 
delta-distribution $\lim_{\sigma\ra 0}\phi_\sigma(x)=\delta(x)$. In the limit $\sigma=0$ 
for \eqref{eq:ss_Fisher-KPP_space_scaled} one recovers the standard elliptic (local) 
FKPP steady state problem 
\be
\label{eq:local_Fisher_steady}
0=\Delta u+u(1-u),
\ee
which again has the homogeneous steady-states $u\equiv0$ and $u\equiv 1$. Note 
that the formal limits $\mu=0$ and $\sigma=0$ do not coincide, if the spatial scaling
is disregarded, since
\benn
\text{\eqref{eq:ss_Fisher-KPP_space_scaled} for }\sigma\ra0~\Rightarrow\quad 0=\Delta u+u(1- u)
\qquad \not\leftrightarrow\qquad 
\text{\eqref{eq:ss_Fisher-KPP} for }\mu\ra0~\Rightarrow\quad 0=\Delta u.
\eenn
For the one-dimensional case, the following result for the nonlocal FKPP-equation is known:

\begin{thm}[\cite{BerestyckiNadinPerthameRyzhik}, $d=1$]
\label{thm:french_thm1}
Suppose the assumptions \eqref{eq:kernel_assume} hold. There exists $\sigma_0>0$ ($\mu_0>0$) 
such that for $\sigma\in[0,\sigma_0]$ ($\mu\in(0,\mu_0]$) 
the only bounded non-negative classical solutions of the stationary nonlocal FKPP equation 
\eqref{eq:ss_Fisher-KPP_space_scaled} (respectively \eqref{eq:ss_Fisher-KPP}) 
are $u\equiv0$ and $u\equiv 1$.
\end{thm}

Theorem \ref{thm:french_thm1} is a persistence result which shows that if the nonlocal effect
is sufficiently small then there are no additional non-negative bounded solutions beyond the two
trivial ones. 
The proof by Berestycki et {al.}~\cite{BerestyckiNadinPerthameRyzhik} uses a combination of a-priori 
estimates, explicit Taylor expansion, approximation on finite regions and several integral 
estimates. 

In this paper we provide a result similar to Theorem \ref{thm:french_thm1} for arbitrary dimensions.
We also lift some assumptions on the kernel $\phi$. This main result is stated in Section 
\ref{sec:result}. We note that our proof does not use the approach in 
\cite{BerestyckiNadinPerthameRyzhik}. We use a perturbation technique involving 
the implicit function theorem in suitable function spaces, bifurcation theory and knowledge 
about the limiting equation \eqref{eq:local_Fisher_steady} for $\sigma=0$; the strategy of the 
proof is outlined in Section \ref{sec:result}.  

%%%%%%%%%%%%%%%%%%%%%%%%%%%%%
\section{The Main Result}
\label{sec:result}

Instead of the assumptions \eqref{eq:kernel_assume} we shall require that 

\begin{itemize}
 \item[(A)] $\phi\in L^1(\R^d)$, $\phi\geq 0$ with $\int_{\R^d} \phi(x) ~\txtd x=1$.
\end{itemize}

We note that some of the assumptions \eqref{eq:kernel_assume} have also been
removed in \cite{HamelRyzhik} for the case $d=1$. However, to remove any form of integrability
or boundedness assumption of the kernel $\phi$ seems to be very difficult, if not 
impossible.
 
\begin{thm}[$d\in\N$]
\label{thm:main}
Suppose (A) holds. For a positive constant $K>1$, there exists a $\sigma_0>0$ ($\mu_0>0$) such that for
$\sigma\in[0,\sigma_0]$ ($\mu\in(0,\mu_0]$) the only bounded
non-negative classical solutions $u\in C^2_b(\R^d)$ with $0\leq \|u\|_{C^2_b}\leq K$ 
of the stationary nonlocal FKPP equation \eqref{eq:ss_Fisher-KPP_space_scaled} 
(respectively \eqref{eq:ss_Fisher-KPP}) are $u\equiv 0$ and
$u\equiv 1$.
\end{thm} 

We point out that in Theorem \ref{thm:main} the constant 
$\sigma_0$ (respectively $\mu_0>0$) does depend upon $K>1$. Indeed, the main difference of Theorem \ref{thm:main}
compared to Theorem \ref{thm:french_thm1} is that our perturbation approach allows for the existence of
certain solution branches $(u_j,\sigma_j)$, where $u_j$ depends upon $\sigma_j$, and a certain weighted
norm of $u_j$ becomes unbounded as $\sigma_j\ra 0$ for $j\ra +\I$. In fact, the restriction on $\|u\|_{C^2_b}$
can be weakened to a weighted Sobolev norm bound described below.\medskip

We outline, on a formal level, the main steps of the perturbation argument:

\begin{enumerate}
 \item[(S1)]\label{S1} \textit{The local problem:} The case $\sigma=0$ is well understood and we collect
 the relevant results later in this section. In particular, it is known that the only bounded
 non-negative solutions of the (local) FKPP equation \eqref{eq:local_Fisher_steady} are $u\equiv 0$
 and $u\equiv 1$.
 \item[(S2)]\label{S2} \textit{Function spaces:} To analyze the problem we utilize spaces for 
 weak solutions to infer results about classical solutions. In particular, one would like to choose Banach spaces 
 which include the homogeneous stationary solutions and are adapted to the linear and nonlinear
 parts of the mapping $F(u,\sigma)$ induced by the nonlocal FKPP equation 
 \eqref{eq:ss_Fisher-KPP_space_scaled}. We use suitably weighted Sobolev spaces and their 
 intersections in this paper; see Section \ref{sec:func_spaces}.
 \item[(S3)]\label{S3} \textit{Regular points:} It turns out that the solution $(u,\sigma)=(1,0)$ can be 
 viewed as a regular point of $F$, {i.e.}~the Frech\'{e}t derivative $(\txtD_uF)_{(1,0)}$ is invertible as a 
 linear map in the spaces chosen in (S2). Then the implicit function theorem shows that the solution branch
 $(u,\sigma)=(1,0)$ is locally unique for sufficiently small $\sigma$; see Section \ref{sec:implicit}.
 \item[(S4)]\label{S4} \textit{Special points:} The solution $(u,\sigma)=(0,0)$ cannot be treated directly
 using the implicit function theorem. Hence it requires a special technique. We use results about 
 purely oscillatory solutions to show that any possible bifurcating solutions near the special 
 point must change sign; see Section \ref{sec:zero}.
 \item[(S5)]\label{S5} \textit{Additional branches:} Using an argument from bifurcation theory we show 
 that there are no additional branches of non-negative bounded solutions outside of neighborhoods of  
 the two states $(u,\sigma)=(0,0)$ and $(u,\sigma)=(1,0)$ for sufficiently small $\sigma_0>0$; see also 
 Figure \ref{fig:01}(a)-(b). The details of this argument can be found in Section \ref{sec:proof_main}. 
 \end{enumerate}

It is important to point out that the steps (S1)-(S5) have been designed with a view towards other
persistence problems arising in nonlocal evolution equations; this extension is discussed in Section 
\ref{sec:discussion}. Let us also remark that the main technical complications arise due to the 
unbounded domain $\R^d$ and the convolution term $\phi_\sigma \ast u$ which substantially restrict the type 
of spaces one may use in (S2).

\begin{figure}[htbp]
\psfrag{a}{\scriptsize{(a)}}
\psfrag{b}{\scriptsize{(b)}}
\psfrag{c}{\scriptsize{(c)}}
\psfrag{d}{\scriptsize{(d)}}
\psfrag{u}{$\|u\|_X$}
\psfrag{u1}{$u\equiv1$}
\psfrag{u0}{$u\equiv0$}
\psfrag{sigma}{$\sigma$}
	\centering
		\includegraphics[width=1\textwidth]{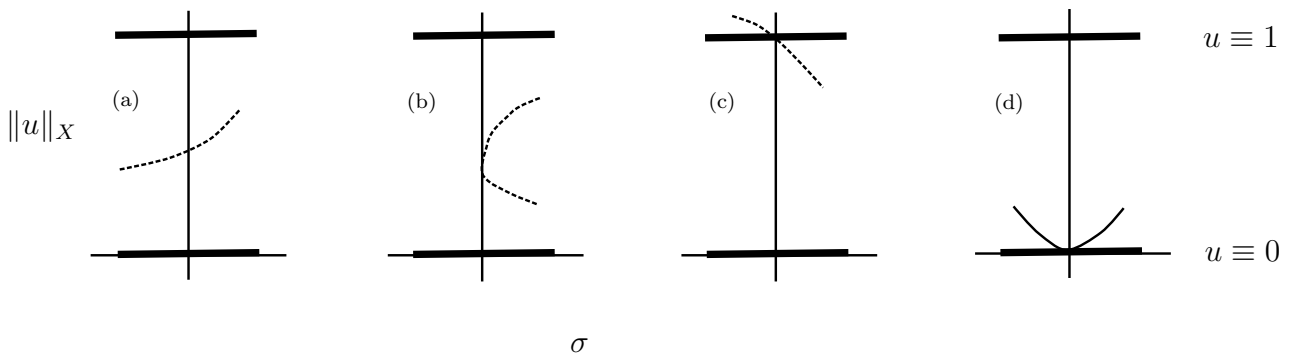}
		\caption{\label{fig:01}Sketch of possible bifurcation scenarios in $(\|u\|_X,\sigma)$-space. The thick
		lines mark the homogeneous stationary states $u\equiv 0$ and $u\equiv 1$. The thin curves indicate possible 
		bifurcation curves. The cases (a)-(b) are impossible as they would violate the result for the local 
		problem $\sigma=0$. The case (c) near $(u,\sigma)=(1,0)$ will be shown to be impossible using the implicit 
		function theorem. We are going to show that the only solutions that could potentially bifurcate near 
		$(u,\sigma)=(0,0)$ are solutions which change sign.}
\end{figure} 

Regardless of these complications, it is always key to understand the step (S1) for a perturbation 
argument. In particular, consider the PDE
\be
\label{eq:KongShen}
0=\Delta u+uf(u),\qquad x\in\R^d~,u:\R^d\ra \R,x\mapsto u(x).
\ee
Suppose $f:\R\ra \R$ is $C^2$. The following result is known:

\begin{thm}(see {e.g.}~\cite{BerestyckiHamelRoques,KongShen})
\label{thm:local}
Suppose $f(u)<0$ for all $u\in\R^+$ with $u\geq \beta_0$ for some $\beta_0>0$ and $f'(u)<0$ for all $u\in\R^+=(0,+\I)$. 
Then there exists a unique positive solution $u^*\in C^2(\R^d,\R^+)$ for \eqref{eq:KongShen} such 
that $\inf_{x\in\R^d} u^*(x)>0$.
\end{thm}

Further variations and generalizations of the previous result are also known \cite{BerestyckiHamelNadirashvili,Rossi}. 
For $f(u)=1-u$ one may just apply Theorem \ref{thm:local} with $\beta_0=1$ -- see also \cite{DuMa} -- which yields the following:

\begin{lem}
\label{lem:local}
The only non-negative solutions of \eqref{eq:local_Fisher_steady} are $u\equiv 0$ and $u\equiv 1$.
\end{lem}

In particular, consider a bifurcation diagram in the space $\{(\sigma,\|u\|_X)\}$, where the norm 
$\|\cdot\|_X$ is defined in the next section, as shown in Figure \ref{fig:01}. The idea is to show that there cannot
be any solution branches tangent to or crossing the vertical segments $\{0\}\times (0,1)$ and $\{0\}\times(1,\I)$. 
However, this does not exclude bifurcating solution branches from $(\sigma,u)=(0,0)$ or $(\sigma,u)=(0,1)$.

%%%%%%%%%%%%%%%%%%%%%%%%%%%%%
\section{Function Spaces}
\label{sec:func_spaces}

To analyze the local behavior of solutions one possibility is to consider the problem via an abstract
nonlinear map. Let $X$, $Y$ be Banach spaces and $I=[0,\sigma_0]\subset\R$ be an interval for some 
$\sigma_0>0$ chosen sufficiently small. Then the mapping $F:X\times I\ra Y$ given by  
\be
\label{eq:FKPP_map}
F(u,\sigma)=\Delta u+u(1-\phi_\sigma\ast u)
\ee
has as a zero set $\{(u,\sigma)\in X\times I:F(u,\sigma)=0\}$ the stationary solutions in $X$ of the nonlocal 
FKPP equation. A choice of function spaces $X$ and $Y$ is required to carry out the analysis explicitly. 
It seems natural to consider Sobolev spaces $W^{k,p}(\R^d)$ for $X$ and $Y$. 
However, $u(x)\equiv 1$ and other nonzero constants do not belong to $W^{k,p}(\R^d)$. Since we 
eventually want to compute a (Fr\'{e}chet) derivative of $F$ at $(u,\sigma)=(1,0)$ the standard 
Sobolev spaces do not suffice. Another natural option would be H\"{o}lder spaces $C^{k+\gamma}(\R^d)$ 
for $k\in\N_0$ and $\gamma\in(0,1)$. The norm in $C^{k+\gamma}(\R^d)$ is given by
\benn
\|u\|_{C^{k+\gamma}(\R^d)}=\sum _{|\alpha|\leq k} \|\txtD^\alpha u\|_\I+\left<u\right>_{C^{k+\gamma}(\R^d)},
\eenn
where $\alpha=(\alpha_1,\alpha_2,\ldots,\alpha_d)$ is a multi-index with $\alpha_i\in \N_0$, 
$|\alpha|=\sum_{i=1}^d \alpha_i$ and 
\benn
\|\txtD^\alpha u\|_\I=\sup_{x\in\R^d} |\txtD^\alpha u(x)|,\qquad \left<u\right>_{C^{k+\gamma}(\R^d)}=\sum_{|\alpha|=k} 
\sup_{x\neq y} \frac{|\txtD^\alpha u(x)-\txtD^\alpha u(y)|}{|x-y|^\gamma}.
\eenn
$C^{k+\gamma}(\R^d)$ does contain the constants but in this case there are problems with the 
convolution term $\phi_\sigma \ast u$ which leads to complications regarding continuity properties of 
the mapping $F$. In this paper we use weighted Sobolev spaces to avoid these problems. However, there
could certainly be other good choices which we do not consider here.\medskip  

Let $w\in L^1(\R^d)$ denote a positive weight function and define the weighted Sobolev space 
$W^{k,p}(\R^d;w)$ for $p\in(1,\I)$ as
\benn
W^{k,p}(\R^d;w)=\{u:w^{1/p} (\txtD^{\alpha}u)\in L^p(\R^d)~\text{for all $|\alpha|\leq k$}\}.
\eenn
With the norm
\benn
\|u\|_{k,p;w}:=\left(\sum_{|\alpha|\leq k}\int_{\R^d} |\txtD^\alpha u(x)|^p w(x)~\txtd x\right)^{1/p}
\eenn
the space $W^{k,p}(\R^d;w)$ is a Banach space \cite{Kufner}. Note that for $k=0$ we just have 
$L^p(\R^d;w):=W^{0,p}(\R^d;w)$ with norm $\|\cdot\|_{p;w}$. If $p=2$ we shall also 
use the standard notation $H^{k}(\R^d;w):=W^{k,2}(\R^d;w)$ for the Hilbert space case. 
We start by fixing the main spaces
\be
\label{eq:spaces}
X=W^{k+2,p}(\R^d;w_X)\cap W^{k,2p}(\R^d;w_X)\qquad \text{and}\qquad Y=W^{k,p}(\R^d;w_Y)
\ee
for $k\in\N_0$ and the norm on $X$ is given by 
\benn
\|u\|_X=\max\{\|u\|_{k+2,p;w_X},\|u\|_{k,2p;w_X}\}.
\eenn
With a view towards bounded classical solutions, we observe that one always has $C^2_b(\R^d)\subset X$. 
For concreteness, we make a choice of weights as well as exponent $p$.
\begin{itemize}
 \item[(B)] Let $\|\cdot\|$ denote the usual Euclidean norm on $\R^d$ and let $l_X$, $l_Y$ be 
 constants such that $\frac12<l_X\leq l_Y<\I$. For the spaces $X$ and $Y$ make the choice
 \be
 \label{eq:weights}
 p=2,\qquad w_X(x)=\frac{1}{(1+\|x\|^2)^{l_X}},\qquad w_Y(x)=\frac{1}{(1+\|x\|^2)^{l_Y}}.
 \ee
\end{itemize}

For \eqref{eq:weights} we have the pointwise estimate $w_Y(x)\leq w_X(x)$ for 
all $x\in\R^d$ which implies that one may order the associated norms 
\be
\label{eq:weights_ineq}
\|u\|_{k,p;w_Y}\leq  \|u\|_{k,p;w_X}. 
\ee
The choice of spaces requires some explanation. To use the intersection of two weighted spaces for $X$ is convenient 
to gain continuity of $F$ as the nonlinear term $u(\phi_\sigma \ast u)$ is essentially a product of two terms. 
The particular choice (B) is going to allow us to use some known results about $H^{k+2}(\R^d;w_X)$ which
are helpful to shorten the proof. Appendix \ref{ap:direct} contains a sketch how one may approach the case 
\eqref{eq:spaces} for more general $p$, $w_X$ and $w_Y$ and avoid using results about $H^{k+2}(\R^d;w_X)$. It
seems important to include this outline as we shall explain in the remarks after Proposition \ref{prop:L1}. 

\begin{remark}
The choice (B) is an auxiliary tool for the analysis and the proof of Theorem \ref{thm:main}. Essentially 
the idea is analogous to existence and regularity 
theory for elliptic equations with smooth input data where weak solution spaces are an analytic tool but
do not appear in the final result.
\end{remark}
  
Before we establish the continuity and differentiability properties of \eqref{eq:FKPP_map} one may check that 
$F:X\ra Y$ is indeed a well-defined map.

\begin{lem}
\label{lem:well-defined1}
Suppose (A)-(B) hold then the map $F:X \times I \ra Y$ is well-defined, {i.e.}, if 
$(u,\sigma)\in X\times I$ then $F(u,\sigma)\in Y$.
\end{lem}

\begin{proof}
By the triangle inequality and \eqref{eq:weights_ineq} we have 
\benn
\|F(u,\sigma)\|_Y=\|\Delta u+u(1-\phi_\sigma\ast u)\|_{k,p;w_Y}\leq  
\|\Delta u\|_{k,p;w_X}+\|u\|_{k,p;w_X}+\|u(\phi_\sigma\ast u)\|_{k,p;w_X}
\eenn
and the first two terms are finite since $u\in W^{k+2,p}(\R^d;w_X)$. For the third term we have
\beann
\|u(\phi_\sigma\ast u)\|^{p}_{k,p;w_X}&=& \sum_{|\alpha|\leq k} \|(\txtD^\alpha u)(\phi_\sigma\ast u)+u(\phi_\sigma\ast \txtD^\alpha u)\|^p_{p;w_X}\\
&\leq & \sum_{|\alpha|\leq k} (\|(\txtD^\alpha u)(\phi_\sigma\ast u)\|_{p;w_X}+\|u(\phi_\sigma\ast \txtD^\alpha u)\|_{p;w_X})^p.
\eeann
Furthermore, by Cauchy-Schwarz it follows that
\bea
\|(\txtD^\alpha u)(\phi_\sigma\ast u)\|_{p;w_X}&\leq&\|\txtD^\alpha u\|_{2p;w_X} ~\|\phi_\sigma\ast u\|_{2p;w_X},\label{eq:bound1}\\
\|u(\phi_\sigma\ast \txtD^\alpha u)\|_{p;w_X}  &\leq&\|u\|_{2p;w_X} ~\|\phi_\sigma\ast \txtD^\alpha  u\|_{2p;w_X}\label{eq:bound2}.
\eea
Upper bounds for the terms $\|\phi_\sigma\ast u\|_{2p;w_X}$ and $\|\phi_\sigma\ast \txtD^\alpha  u\|_{2p;w_X}$ from 
\eqref{eq:bound1}-\eqref{eq:bound2} can be obtained by using the generalized Young's inequality from Appendix \ref{ap:LiebLoss}
and the assumption $\|\phi\|_1=1$ from (A) 
\bea
\|\txtD^\alpha u\|_{2p;w_X} ~\|\phi_\sigma\ast u\|_{2p;w_X}& \leq &\|\txtD^\alpha u\|_{2p;w_X} ~
~\| u\|_{2p;w_X},\label{eq:bound3}\\
\|u\|_{2p;w_X} ~\|\phi_\sigma\ast \txtD^\alpha  u\|_{2p;w_X}  &\leq& 
\|u\|_{2p;w_X} ~\| \txtD^\alpha  u\|_{2p;w_X}\label{eq:bound4}.
\eea 
Since $u\in X$ by assumption the result follows.
\end{proof}

\begin{lem}
\label{lem:Frechet}
Suppose (A)-(B) hold then the mapping $F:X\times I\ra Y$ is continuous in $\sigma$ and 
continuously differentiable in $u$ with Fr\'{e}chet derivative
\benn
(\txtD_uF)_{(u,\sigma)}U=\Delta U+(1-\phi_\sigma\ast u)U-u(\phi_\sigma\ast U).
\eenn
\end{lem}

\begin{proof}
For continuity in $\sigma$ we just focus on continuity at $\sigma=0$, the case $\sigma>0$ is easily checked. We have
\beann
\|F(u,\sigma)-F(u,0)\|_Y^p&=&\|u(\phi_\sigma\ast u-u)\|_{k,p;w_Y}^p\\
&\leq& \sum_{|\alpha|\leq k}(\|\txtD^\alpha u(\phi_\sigma\ast u-u)\|_{p;w_X}
+\|u(\phi_\sigma\ast \txtD^\alpha u-\txtD^\alpha u)\|_{p;w_X})^p.
\eeann
Similar to the proof of Lemma \ref{lem:well-defined1}, the Cauchy-Schwarz inequality yields that 
\bea
\|\txtD^\alpha u(\phi_\sigma\ast u-u)\|_{p;w_X}&\leq &\|\txtD^\alpha u\|_{2p;w_X} \|\phi_\sigma\ast u-u\|_{2p;w_X},\\
\|u(\phi_\sigma\ast \txtD^\alpha u-\txtD^\alpha u)\|_{p;w_X}&\leq & 
\|u\|_{2p;w_X} \|\phi_\sigma\ast \txtD^\alpha u-\txtD^\alpha u\|_{2p;w_X}.
\eea
Recall that there is strong convergence $\phi_\sigma\ast v\ra v$ for $v\in L^r(\R^d)$ with $1\leq r<\I$ \cite[p.64]{LiebLoss}
as $\sigma\ra 0$; the proof can be adapted to yield strong convergence in $L^r(\R^d;w_X)$ as shown in Appendix \ref{ap:LiebLoss}. 
Setting $r=2p$ yields that $\|F(u,\sigma)-F(u,0)\|_Y\ra 0$ as $\sigma\ra 0$ which yields continuity at $\sigma=0$. Next, 
fix $\sigma \in I$ and let $u,v\in X$ then we obtain for continuity in $u$ that
\benn
\|F(u,\sigma)-F(v,\sigma)\|_Y\leq \|\Delta (u-v)\|_Y+ \|u-v\|_Y+ \|u (\phi_\sigma\ast u)-v(\phi_\sigma\ast v)\|_Y,
\eenn
where the first two terms in the right-hand side can be made small since $X\subset Y$. For the last term 
a direct calculation shows that
\benn
\|u (\phi_\sigma\ast u)-v(\phi_\sigma\ast v)\|_Y\leq \|(u-v)(\phi_\sigma\ast u)\|_Y+\|v(\phi_\sigma\ast(u-v))\|_Y,
\eenn
where both summands can be made small by using Cauchy-Schwarz and the generalized Young inequality as in Lemma 
\ref{lem:well-defined1} which implies continuity in $u$. Calculating the G\^{a}teaux derivative in $u$ yields
\benn
(\nabla_uF)[U]= \lim_{\epsilon\ra 0}\frac1\epsilon[F(u+\epsilon U,\sigma)-F(u,\sigma)]
=\Delta U-u(\phi_\sigma\ast U)+(1-\phi_\sigma\ast u) U.
\eenn 
To show that the G\^{a}teaux derivative coincides with the Fr\'{e}chet derivative, {i.e.}~$\nabla_uF=\txtD_uF$, we have to
verify continuity of $\nabla_uF$ in $u$ \cite[p.47]{Deimling}. We have
\benn
\|(\nabla_uF)[U]-(\nabla_vF)[U]\|_Y\leq \|(u-v)(\phi_\sigma\ast U)\|_Y+\|(\phi_\sigma\ast (u-v))U\|_Y
\eenn
and the same argument as for proving continuity of $F$ in $u$ can be applied.
\end{proof}

Based on the proof it is now more evident why the construction of $X$ is essentially enforced by the nonlinear
structure of the nonlocal FKPP equation. For example, for $k=0$ and $p=2$ the quadratic term $u^2$ for $\sigma=0$
indicates that the space $L^4(\R^d;w_X)$ would be a good choice. Observe also that one does not have 
to make the precise choice (B) to prove Lemma \ref{lem:Frechet}. In fact, what is required is the 
weaker assumption that there exists a positive constant $K>0$, possibly dependent upon $p$ and $k$, such that
\be
\label{eq:weights_ineq1}
\|u\|_{k,p;w_Y}\leq K \|u\|_{k,p;w_X} 
\ee
for a choice of weights for which Young's inequality still holds. Hence the key parts for continuity 
and differentiability properties of $F$ in Lemma \ref{lem:Frechet} do not depend upon the concrete 
choice \eqref{eq:weights} but only on constructing an intersection of weighted Sobolev spaces for $X$ adapted to
the nonlinearity with sufficiently 'nice' weights. 

%%%%%%%%%%%%%%%%%%%%%%%%%%%%%
\section{The Implicit Function Theorem}
\label{sec:implicit}

The next goal is to apply the implicit function theorem; see \cite[p.~148]{Deimling} for a detailed 
statement. In particular, we want to consider neighborhoods of $(u,\sigma)=(1,0)$ and 
$(u,\sigma)=(0,0)$. Since $w\in L^1(\R^d)$ the non-zero constants belong to $X$. Substituting the 
stationary solutions $(u,\sigma)=(1,0)$ and $(u,\sigma)=(0,0)$ into Lemma \ref{lem:Frechet} yields:

\begin{lem} 
The linearized operators acting on $U\in X$ with $U=U(x)$, $x\in\R^d$, are
\bea
\cL_0:=(\txtD_uF)_{(0,0)}U&=&\Delta U+U\label{eq:lin_at_0},\\
\cL_1:=(\txtD_uF)_{(1,0)}U&=&\Delta U-U\label{eq:lin_at_1}.
\eea
\end{lem}

For the one-dimensional case $d=1$, one gets the ordinary differential equations (ODEs)
\benn
(\textnormal{ODE})_0:\left\{ \begin{array}{lcl} U'&=&V,\\ V'&=&-U,\\\end{array}\right.\qquad \text{and}\qquad 
(\textnormal{ODE})_1:\left\{ \begin{array}{lcl} U'&=&V,\\ V'&=&U.\\\end{array}\right.
\eenn
Clearly, $(\textnormal{ODE})_1$ has a saddle-point at the origin and hence $U''-U=0$ has no other bounded solutions 
except $U\equiv 0$. However, $(\textnormal{ODE})_0$ has a center equilibrium with infinitely many bounded non-zero 
periodic solutions. This indicates that $\cL_0$ does not have the required inverse to apply the implicit 
function theorem. Hence we are going to treat the neighborhood of the zero branch $u\equiv 0$ separately
in Section \ref{sec:zero}.\medskip 

For $\cL_1$ on $\R^d$ it is tempting to consider the Fourier transform
\benn
\hat{u}(\xi):=\frac{1}{(2\pi)^{d/2}}\int_{\R^d} \txte^{-ix^T \xi}u(x)~\txtd x,\qquad \xi\in\R^d,
\eenn
where $(\cdot)^T$ denotes the transpose and apply it to $\cL_1U=0$. This yields
\be
\label{eq:Evans}
(1+\|\xi\|^2)\hat{U}(\xi)=0, 
\ee
where $\|\cdot\|$ denotes the Euclidean norm. From \eqref{eq:Evans} it follows that $\hat{U}=0$ which
implies that the nullspace of $\cL_1$ only contains the zero solution. However, this calculation assumes that $u\in L^p(\R^d)$
for some suitable $p$, {e.g.}~$p=1$ or $p=2$, whereas we have to work in weighted spaces. To illustrate the problem, 
we consider the one-dimensional case which can be solved explicitly. If $d=1$ we have $0=\cL_1U=U''-U$ so that 
the general solution is
\benn
U(x)=c_1\txte^{-x}+c_2\txte^{x}\qquad \text{for constants $c_1,c_2\in\R$.}
\eenn
If we would choose $w(x)=\txte^{-x^2}$ it follows that for constants $\alpha_1 \in\R$, $\alpha_2\in(0,\I)$, we also have
\benn
\int_\R \txte^{\alpha_1 x}\txte^{-\alpha_2 x^2}~\txtd x=\txte^{\frac{\alpha_1^2}{4\alpha_2}}\sqrt{\frac{\pi}{\alpha_2}}<\I.
\eenn
For the case $p=2$, $k=0$ the last calculation yields that $\cL_1U=0$ has non-zero solutions in 
$H^2(\R;\txte^{-x^2})\cap L^{4}(\R;\txte^{-x^2})$ which implies that $\cL_1:X\ra Y$ is
not invertible. For the case $w(x)=(1+x^2)^{-1}$ it is straightforward to calculate that the integration 
\be
\label{eq:1D_base}
\int_\R \frac{U(x)^2}{1+x^2}~\txtd x=\int_\R \frac{(c_1\txte^{-x}+c_2\txte^{x})^2}{1+x^2}~\txtd x
\ee
implies that $U(x)w^{1/2}(x)\in L^2(\R)$ if and only if $c_1=0=c_2$. Therefore, $\cL_1U=0$ on 
$H^2(\R;(1+x^2)^{-1})\cap L^4(\R;(1+x^2)^{-1})$ if and only if $U\equiv0$ which
implies that $\cL_1$ has trivial nullspace when $w(x)=(1+x^2)^{-1}$ is used. Hence, the choice of weight
function is crucial if we want to apply the implicit function theorem at the constant solution $(u,\sigma)=(1,0)$.\medskip

A natural strategy is to prove that $\cL_1:X\ra Y$ has trivial nullspace for suitable classes of $w_X$, $w_Y$, $p$ and
then show that $\text{nullspace}(\cL_1)=\{0\}$ implies that $\cL_1$ is invertible in a rather large class of weighted spaces. 
For our choice of $w_X$, $w_Y$, $p$ 
given in (B) one may directly use previous results.

\begin{prop}
\label{prop:L1}
For the choice \eqref{eq:weights}, {i.e.}, under assumption (B)
\be
\label{eq:weights1}
p=2,\qquad w_X(x)=\frac{1}{(1+\|x\|^2)^{l_X}},\qquad w_Y(x)=\frac{1}{(1+\|x\|^2)^{l_Y}},
\ee
with $\frac12<l_X\leq l_Y<\I$, the operator $\cL_1:X\ra Y$ is invertible with bounded inverse.
\end{prop}

\begin{proof}
Apply \cite[{Thm.}~4.2, see p.58]{GindikinVolevich} which makes the argument based upon the Fourier
transform precise for the weighted spaces defined via \eqref{eq:weights1}; the result \cite[{Thm.}~4.2, see p.58]{GindikinVolevich} 
uses that the symbol in \eqref{eq:Evans} has no zeros which is easily checked. 
\end{proof}

It is important to note that we would have to establish similar results as Proposition \ref{prop:L1} for persistence
results in other types of nonlocal PDEs again, as the linearized problem may change. Therefore, we provide an outline 
in Appendix \ref{ap:direct} how invertibility can be established via a more direct approach. 

%%%%%%%%%%%%%%%%%%%%%%%%%%%%%
\section{The Zero Solution}
\label{sec:zero}

As discussed in the previous section, one has to treat $\cL_0$ separately. First, we apply Taylor's 
Theorem (in $u$) to $F:X\ra Y$, which yields
\be
\label{eq:Taylor1}
F(u,\sigma)=F(0,\sigma)+\txtD_uF_{(0,\sigma)}u+R(u),
\ee
where the remainder $R(u)$ satisfies $R(u)=\cO(\|u\|_X^2)$. Note that for the nonlocal FKPP equation
the remainder is exact and given by $R(u)=-u(\phi_\sigma\ast u)$. However, we shall only need the 
existence of a bound for the remainder when $\|u\|_X$ is small; one may calculate this bound efficiently
also for other equations, not only FKPP, by using the Lagrange or integral forms of the remainder 
(see {e.g.}~\cite[{Ch.5}]{DudleyNorvaisa}). Since we always have 
$F(0,\sigma)=0$ it follows from \eqref{eq:Taylor1} that 
\be
\label{eq:Taylor2}
F(u,\sigma)=\txtD_uF_{(0,\sigma)}u+R(u)=\Delta u+u+R(u).
\ee
Let $Z:=X\cap L^\I(\R^d)$ with norm $\|\cdot\|_Z=\max\{\|\cdot\|_X,\|\cdot\|_\I\}$ and 
observe that we may make the last term in \eqref{eq:Taylor2} small if $u$ is in a 
neighborhood of $u=0$ {i.e.}~when $\|u\|_Z$ is small. We start by proving a one-dimensional 
result to illustrate, why it is expected that a change-of-sign plays a role for solutions
close to zero.

\begin{prop}
Let $d=1$ then there exists $\sigma_0>0$ and a ball $B\subset Z$ centered at $u=0$ 
such that $B\times [0,\sigma_0]$ does not contain any non-negative bounded classical solution 
$u\not\equiv 0$, $u\in C^2_b$, to $F(u,\sigma)=0$.
\end{prop}

\begin{proof}
Fix $\epsilon>0$. Consider any non-negative bounded classical solution 
$(u,\sigma)\in B\times (0,\sigma_0]$ with $u\not\equiv 0$ and define $\eta:=\phi_\sigma\ast u$ 
then we obtain
\benn
F(u,\sigma)=u''+u(1-\eta),
\eenn
where $|\eta(x)|<\epsilon$ holds for all $x\in \R$ as $\|u\|_Z$ can be 
made small; note that this step would have worked with a general small remainder for a
Taylor expansion. Considering $F(u,\sigma)=0$ leads to the two-dimensional non-autonomous vector
field
\be
\label{eq:ODE_1d}
\left\{ \begin{array}{lcl} \frac{\textnormal{d}u}{\textnormal{d}x}(x)&=&v(x),\\ 
\frac{\textnormal{d}v}{\textnormal{d}x}(x)&=&-u(x)(1-\eta(x)).\\\end{array}\right.
\ee
We argue by contradiction and suppose there exists a nonzero solution $(u,v)$ to 
\eqref{eq:ODE_1d} such that $u(x)\geq 0$ for all $x\in\R$ and there exists some $x$ where
$u(x)>0$. We have the pointwise estimate 
\be
\label{eq:good_estimate}
-u(x)(1+\epsilon)<v'(x)<-u(x)(1-\epsilon)
\ee
for all $x\in\R$. A phase plane analysis can now be carried out. Since $u(x)>0$ for some 
$x\in \R$ we may take this value as $x=0$ without loss of generality and consider
\benn
(u(0),v(0))\in \left\{[0,\I)\times (-\I,\I)\right\}-\{(0,0)\}. 
\eenn
Furthermore, if $v(0)\leq 0$ then we immediately get that $v(x)<0$ for
some $x>0$ close to zero. If $v(x)<0$ it follows that $u$ will decay until we have
$u(x^*)<0$ for some $x^*>0$ which yields a contradiction. Hence, we are left 
with the case $v(0)>0$. However, if $v(x)>0$ then $u$ increases so the estimate
\eqref{eq:good_estimate} can be applied to show that $v$ has to cross $\{v=0\}$ into
the region where $v(x)<0$. This finishes the proof.
\end{proof}

The key ideas of the last proof are that we just look at solutions inside a small
neighborhood of the origin $(u,\sigma)=(0,0)$ in $X\times I$ and that any nonzero 
solution $u$ must change sign for an equation which is controlled by solutions of  
\benn
u''+u(1\pm \epsilon)=0. 
\eenn
To generalize this idea we briefly recall some classical theory for linear homogeneous 
partial differential operators \cite{Hoermander1,Hoermander2} 
focusing just on the operator given by
\benn
\cL^\alpha_0U:=\Delta U+(1+\alpha)U,\qquad \alpha\in\R,~|\alpha|<\epsilon
\eenn
for sufficiently small fixed $\epsilon$ with $0<\epsilon<1$. 

\begin{remark}
We are going to develop the following ideas in slightly more generality than 
strictly necessary here to illustrate that the approach is applicable to a much broader 
class of problems. 
\end{remark}

A solution $U=U(x)$ of $\cL^\alpha_0U=0$ is called an exponential solution if it is of the form
\be
\label{eq:H_sol}
U(x)=f(x)\txte^{\txti\xi^Tx},\qquad \text{for $\xi\in\C^d$,} 
\ee
where $f$ is a polynomial. We say that a solution of the form \eqref{eq:H_sol} is a simple exponential solution
if $f\equiv 1$ and $\xi\neq 0$. Furthermore, we call a simple exponential solution purely oscillatory if $\xi\in\R^d$. 
The next result is a slightly modified version of \cite[{Thm.}~7.3.6, p.185]{Hoermander1}.

\begin{prop}
\label{prop:Hoermander_sol}
The closed linear hull of simple exponential solutions to $\cL^\alpha_0U=0$ in the space $C^\I(\R^d)$
yields all solutions to $\cL_0^\alpha U=0$ in $C^\I(\R^d)$.
\end{prop}

\begin{proof}
To show that the closed linear hull of exponential solutions contains all solutions follows verbatim from  
\cite[{Thm.}~7.3.6]{Hoermander1} as $\cL^\alpha_0$ is a special case of the elliptic operators covered. To 
restrict the class to simple exponential solutions, one observes that the symbol of $\cL^\alpha_0$ has no 
multiple factors which makes the remark in \cite[p.39]{Hoermander2} applicable such that $f(x)\equiv 1$
for all exponential solutions. Since we restrict to nonzero solutions it follows that $\xi\neq 0$ can be
assumed as well since $\cL^\alpha_0U=0$ has no other constant solutions except $U\equiv 0$. 
\end{proof}

The next result is not necessary to prove the main result of this paper but it is very interesting
to understand the structure of solutions to $\cL^\alpha_0U=0$ in weighted spaces. 

\begin{prop}
\label{prop:oscillatory}
Suppose (B) holds with weight function $w_X\in L^1(\R^d)$. Suppose $\cL^\alpha_0U=0$ and 
$U$ is in the linear hull of simple exponential solutions. Then $U\in X$ if and only if 
$U$ is in the linear hull of purely oscillatory solutions.
\end{prop}

\begin{proof}
See Appendix \ref{ap:osc}.
\end{proof}

The next result, albeit straightforward, is a crucial observation to deal with solutions near
the point $(u,\sigma)=(0,0)$ in the proof of Theorem \ref{thm:main}.

\begin{lem}
\label{lem:change_sign}
There exists a purely oscillatory solution of $\cL^\alpha_0U=0$ such that $U$ changes sign {i.e.}
there exist two points $x^-,x^+\in\R^d$ such that $U(x^-)<0$ and $U(x^+)>0$.
\end{lem}

\begin{proof}
It is easy to see that $U(x)=\cos(\sqrt{1+\alpha}x_1)$. In fact, there are many other 
trigonometric functions which would work as well.
\end{proof}

The last proof also showed that $\cL^\alpha_0:X\ra Y$ is not invertible as it does have a non-trivial 
nullspace; {cf.}~Appendix \ref{ap:direct}. The next step is to exclude the existence of non-negative bounded
solutions $u\not\equiv 0$ for sufficiently small nonlocality near the special point $(u,\sigma)=(0,0)$
for arbitrary dimension $d$.

\begin{prop}
\label{prop:zero_branch}
There exists $\sigma_0>0$ and a ball $B\subset Z$ centered at $u=0$ such that $B\times [0,\sigma_0]$ does not 
contain any non-negative bounded classical solution $u\not\equiv 0$, $u\in C^2_b(\R^d)$, to $F(u,\sigma)=0$.
\end{prop}

\begin{proof}
Fix $\epsilon>0$. Let $u\in B$ and $u\not\equiv 0$. If $u\geq 0$ then it follows $u(x)>0$ for all
$x\in\R^d$ by \cite[{Lem. 2.1}]{BerestyckiNadinPerthameRyzhik}. We argue by contradiction 
and suppose that there exists a solution $u(x)>0$ for all $x\in\R^d$. In particular, this implies 
$u(0)>0$. Observe that 
\be
\label{eq:ineq_comp}
0=\Delta u+u(1-\phi_\sigma \ast u)\geq \Delta u+u(1-\epsilon)
\ee
holds with $0<\epsilon<1$ if $\|u\|_{L^\infty}<\epsilon$, which we can 
achieve by shrinking $B$. Define the differential operator 
\benn
\cL^\epsilon v:=\Delta v+v(1-\epsilon),\qquad \text{with $0<\epsilon<1$}.
\eenn 
Now we choose a special radially symmetric oscillatory 
solution $\bar{u}^\epsilon$ with $\bar{u}^\epsilon(0)>0$ which, for some 
$\delta_3>\delta_2>\delta_1>0$, obeys 
\be
\label{eq:radial_cond1}
\bar{u}^\epsilon(x)>0\text{ for $x\in \{\|x\|<\delta_1\} \cup \{\delta_2<\|x\|\leq\delta_3\}$}
\ee 
and
\be
\label{eq:radial_cond2} 
\bar{u}^\epsilon(x)<0\text{ for $x\in \{\delta_1<\|x\|<\delta_2\}$}.
\ee
To construct such a solution we may just take the solution from the proof Lemma \ref{lem:change_sign}
for dimension $d=1$. For $d\geq 2$, we note that $\cL^\epsilon v=0$ is just Helmholtz's equation and 
a suitable radial solution is given by scaled hyperspherical Bessel functions; see 
{e.g.}~\cite{FengHuangYang}. We may scale the solution $\bar{u}^\epsilon$ by a positive 
constant, say $\kappa>0$, and still denote it $\bar{u}^\epsilon$, such that 
\be
\label{eq:xcondS}
\bar{u}^\epsilon(x)>u(x)\text{ for $x\in \{\delta_2+\delta<\|x\|<\delta_3\}$},
\ee
for some small constant $\delta>0$, and a suitable $\delta_3$ such that $\delta_2+\delta<\delta_3$ and the 
conditions \eqref{eq:radial_cond1}-\eqref{eq:radial_cond2} still hold. By making $\kappa>0$ sufficiently large 
and using smoothness of $u$ and 
$\bar{u}^\epsilon$ it follows that the mapping $(u-\bar{u}^\epsilon):\R^d\ra \R$ has constant rank one for 
$x\in \{\delta_2<\|x\|<\delta_3\}$. Now the constant-rank level set theorem 
\cite[{Thm 8.8}]{Lee} and smoothness of $u$ and $\bar{u}^\epsilon$ yield the existence of a closed 
embedded $(d-1)$-dimensional manifold $\Gamma$, which consists of the zeros of $u-\bar{u}^\epsilon$ in 
$\{\delta_2<\|x\|<\delta_2+\delta\}$, {i.e.}, $u(x)=\bar{u}^\epsilon(x)$ for $x\in \Gamma$. In fact, 
$\Gamma$ is smooth \cite[{Thm 8.2}]{Lee}. 

Furthermore, consider an arbitrary nonzero unit length vector $v\in\R^d$ and the function 
$\rho(\lambda):=(u-\bar{u}^\epsilon)(\lambda v)$ for $\lambda\in(\delta_2,\delta_2+\delta)$. By continuity, there 
exists a point $p=\lambda v\in\Gamma$ such that $\rho(p)=0$.
Upon making $\delta>0$ sufficiently small, and increasing $\kappa>0$ if necessary, it follows that 
there exists a unique such point $p$. Hence, it follows that $\Gamma$ is smooth manifold, which is 
topologically a sphere. Let $\Omega$ denote the bounded domain with boundary $\partial \Omega = \Gamma$ 
inside the topological sphere $\Gamma$, {i.e.},
\benn
\Omega:=\bigcup_{v\in \Gamma} \{\lambda v:\lambda \in[0,1] \}.
\eenn
Using \eqref{eq:ineq_comp} 
we apply the comparison 
principle \cite[{Thm.}~2.2.4]{PucciSerrin} to $\cL^\epsilon$ for the domain $\Omega$, which implies
\be
\label{eq:large_sol}
u(x)\geq \bar{u}^\epsilon(x) 
\ee
for all $x\in\Omega$. However, we can scale $\bar{u}^\epsilon$ by a positive constant so that $\bar{u}^\epsilon(x)>u(x)$ 
for some $x\in \{\|x\|<\delta_1\}$. This contradiction to \eqref{eq:large_sol} concludes the proof.
\end{proof}

In the next proposition, we strengthen the conclusion of Proposition \ref{prop:zero_branch} 
to a ball in $X$.

\begin{prop}
\label{prop:zero_branch_ra}
There exists $\sigma_0>0$ and a ball $B\subset X$ centered at $u=0$ such that $B\times [0,\sigma_0]$ does not 
contain any non-negative bounded classical solution $u\not\equiv 0$, $u\in C^2_b(\R^d)$, to $F(u,\sigma)=0$.
\end{prop}

\begin{proof}
Observe that to establish \eqref{eq:ineq_comp} in Proposition \ref{prop:zero_branch} we needed 
the bound $0<\|u\|_\I<\epsilon <1$. Note that we may not have this uniform $L^\I$-bound on 
$u$ for $u\in B\subset X$, where $B$ is a small ball centered at $u=0$. However, by making $B$ small enough 
and using the assumption that $u\in C^2_b$, there exists a ball
\benn
\cB(2\delta):=\{x\in\R^d:\|x\|\leq 2\delta\}\subset \R^d
\eenn
for some $\delta>0$ such that for a given fixed constant $\epsilon\in(0,1)$
\benn
\delta >\delta_3\quad \text{and}\quad  u(x)<\frac13\epsilon~\text{for $x\in \cB(2\delta)$}.
\eenn
Indeed, suppose $\cB(2\delta)$ does not exist, then there exists an open bounded set $\cM\subset \R^d$ 
of positive volume such that $u(x)\geq\frac13 \epsilon$ for all $x\in \cM$. In this case, we can make $B$ smaller
to obtain a contradiction. This implies that all sets $\cM$ such that $u(x)\geq\frac13\epsilon$ for $x\in \cM$ can be
assumed to exist only outside the ball $\cB(2\delta)$. Hence we deduce that
\beann
\int_{\R^d}\phi_\sigma(x-y)u(y)~\txtd y&=&\int_{\R^d-\cB(2\delta)}\phi_\sigma(x-y)u(y)~\txtd y+
\int_{\cB(2\delta)}\phi_\sigma(x-y)u(y)~\txtd y\\
&\leq & \int_{\R^d-\cB(2\delta)}\phi_\sigma(x-y)u(y)~\txtd y +\frac13\epsilon, 
\eeann
where the first term on the right can be made small if $x\in\cB(\delta)$ and $\sigma>0$ is sufficiently small; 
in particular, we can make it smaller than $\frac23\epsilon$ since for $x\in \cB(\delta)$ and $\sigma$ small, 
the function $\phi_\sigma(x-y)$ is concentrated near $\partial \cB(\delta)$ and $u(x)\leq \frac13\epsilon$ for $x$ near 
$\partial \cB$ as $u\in C_b^2$. Therefore, the inequality \eqref{eq:ineq_comp} 
holds on $\cB(\delta)$. Now we can repeat the argument from 
Proposition \ref{prop:zero_branch} as all relevant sets are contained in $\cB(\delta)$.
\end{proof}

We could have proved Proposition \ref{prop:zero_branch} slightly quicker 
and bypassed some of the preliminary discussion of the oscillatory solution space. However, 
it is important to point out the structure of special points. In 
particular, the linearized problem $\txtD_uF_{(0,0)}u=0$, and a suitable small perturbation of it, only contain
solutions outside of the class we are interested in ('bounded non-negative classical solutions'). This observation 
facilitates the comparison approach and is a much more general strategy applicable to problems beyond 
the FKPP equation.

%%%%%%%%%%%%%%%%%%%%%%%%%%%%%
\section{Proof of the Main Result}
\label{sec:proof_main}

Recall that we try to establish that the only bounded non-negative classical solutions to the nonlocal 
FKPP equation for small nonlocality are $u\equiv 0$ and $u\equiv 1$.

\begin{proof}(of Theorem \ref{thm:main})
By Proposition \ref{prop:L1} and the implicit function theorem applied to $F:X\times I\ra Y$ at 
$(u,\sigma)=(1,0)$ it follows that there exist balls $B(1,r_1)\subset X$ and $B(0,r_0)\subset I$ 
with some radii $r_{0,1}>0$ and exactly one continuous map $T:B(0,r_0)\ra B(1,r_1)$ such that 
\benn
T(0)=1\qquad \text{and}\qquad F(T(\sigma),\sigma)=0\quad\text{on $B(0,r_0)$.}
\eenn   
Since $F(1,\sigma)=0$ gives a solution branch $u\equiv 1$ the uniqueness of $T$ implies that 
$T(\sigma)\equiv1$ for $\sigma\in B(0,r_0)$. Hence, in a sufficiently small neighborhood $\cN_1$ 
of $(u,\sigma)=(1,0)$ only the homogeneous solution $u\equiv1$ exists.

In a neighborhood $\cN_0$ of $(u,\sigma)=(0,0)$ we may apply Proposition \ref{prop:zero_branch_ra} to 
conclude that the only branch of bounded non-negative classical solutions to $F(u,\sigma)=0$ is given by
$u\equiv 0$. Hence it remains to consider the possibility of solutions outside of neighborhoods of 
the two trivial solution branches.

We argue by contradiction. Suppose there exists a sequence of bounded non-negative 
solutions $(u_j,\sigma_j)$ such that 
\benn
(u_j,\sigma_j)\in X\times I-(\cN_0\cup\cN_1),\quad 0\leq \|u_j\|_{C^2_b}\leq K,\quad
\text{ for all $j\in\N$,}
\eenn
and we have 
\benn
F(u_j,\sigma_j)=0,\qquad \sigma_j\ra 0\text{ as }j\ra \I.
\qquad  
\eenn
Note that the sequence $\{u_j\}$ is bounded in $X_1=W^{k+2,p}(\R^d;w_X)$ for $k=0$ and $p=2$, 
{i.e.}, in $X_1=H^{2}(\R^d;w_X)$. Due 
to the theory of weighted Sobolev spaces and their associated compact embeddings \cite[{Prop. 2.5}]{HarringtonRaich} 
we have that the embedding $W^{k+2,2}(\R^d;w_X)\hookrightarrow W^{k+1,2}(\R^d;w_X)$ is compact. It
follows, upon passing to a subsequence, $u_{j}\ra u_\I$ for some $u_\I\in \tilde{X}_1:=H^1(\R^d;w_X)$. By continuity 
of $F$ from Lemma \ref{lem:Frechet}, it follows that $u_\I$ is a weak solution to the local problem
\benn
F(u,0)=0
\eenn
{i.e.}~for any test function $\psi\in C^\I_c(\R^d)$ we have 
\benn
-\int_{\R^d}\sum_{j=1}^d\frac{\partial u_\I}{\partial x_j}(x)  \frac{\partial \psi}{\partial x_j}(x) ~\txtd x+\int_{\R^d}u_\I(x)(1-u_\I(x))\psi(x)~\txtd x. 
\eenn
Furthermore, $u_\I\geq 0$ almost everywhere as $u_j(x)\geq 0$ for all $x\in\R^d$, $j\in\N$. 
By elliptic regularity we can obtain that $u_\I\in C^2(\R^d)$. Therefore, the pair 
$(u_\I,0)\in X\times I-(\cN_0\cup\cN_1)$ yields a non-negative 
solution to the local FKPP equation, which does not coincide with the two constant solutions. 
This contradicts Lemma \ref{lem:local} and the main result follows. 
\end{proof}

%%%%%%%%%%%%%%%%%%%%%%%%%%%%%
\section{Outlook}
\label{sec:discussion}

From the details of the proof in Section \ref{sec:proof_main} it is already evident that we 
have not shown the strongest possible result that our methods allow.
For example, Theorem \ref{thm:local} for $\sigma=0$ holds for a much wider class of 
equations beyond the local FKPP equation. Also standard results about the Laplacian, which we
used, do hold for much more general general elliptic operators. Hence one could rephrase our main theorem
for certain nonlocal PDEs of the form 
\be
\label{eq:gen_nonlocal}
0=A u+f(u,\phi_\sigma\ast u),\qquad x\in\R^d,~u=u(x),
\ee
where $A$ is a suitable generalization of the Laplacian and $f$ is sufficiently smooth
function which generalizes $f(u,\phi_\sigma\ast u)=u(1-\phi_\sigma\ast u)$. However, it 
seemed more accessible to illustrate the overall strategy on an example first. The generalization 
will be considered in future work.\medskip 

We briefly point out how the steps (S1)-(S5) described in 
Section \ref{sec:result} have to be modified in an attempt to obtain results for 
\eqref{eq:gen_nonlocal} or even more general persistence theorems for solutions of nonlocal problems. 
The step (S1) should always be the starting point for a local perturbation argument. In particular, 
the relevant solution spaces should be fully understood for the local problem. (S2) is based
upon a suitable choice of function spaces which are adapted to the nonlinearity of the problem and 
to obtain continuity of the solution mapping. For example, it may be necessary to use other types 
of intersections of two or more weighted Sobolev spaces. Based upon the well-chosen construction
of the spaces in (S2) the application of the implicit function theorem at regular points makes (S3)
a relatively standard step. Nevertheless, if there are no previous results available for the 
invertibility of the class of linear operators under consideration, then a direct argument to show
invertibility is often highly nontrivial; see Appendix \ref{ap:direct}. If special points occur then 
one has to find a modification of -- or alternative to -- the step (S4). 
For our case, the key idea is that the linearized problem at the special point, as well as a small
perturbation of this linear problem, do not contain any non-negative bounded solutions we are 
interested in. Hence we have been able to use a comparison argument to exclude the existence of 
a certain class of solutions near the special point. The step (S5) requires a good understanding
of the local problem and is expected to be quite standard for most problems.\medskip

As the last remark, we stress again that the assumption of a sufficiently small nonlocality
is crucial. On bounded domains it is well-known that the stationary 
solution $u\equiv 1$ of the nonlocal FKPP equation may bifurcate into other bounded non-negative 
solutions for suitably large nonlocality \cite{FurterGrinfeld}. The same result is likely 
to be true for the unbounded domain case. In fact, numerical simulations \cite[{Fig.}1]{BerestyckiNadinPerthameRyzhik} 
naturally lead to the conjecture that the state $u\equiv 1$ may undergo a ('supercritical Hopf-type') 
bifurcation with an exchange 
of stability to oscillations. It is important future work to investigate this bifurcation point in more 
detail as the numerical simulations
\cite[{Fig.}1]{BerestyckiNadinPerthameRyzhik} also show that the associated traveling wave between
$u\equiv 0$ and $u\equiv 1$ changes from an equilibrium-to-equilibrium (E-to-E) heteroclinic
front to an equilibrium-to-periodic (E-to-P) heteroclinic front. In fact, very recently the
existence of oscillatory bounded solutions for \textit{sufficiently 
large} nonlocality has been proven \cite{HamelRyzhik} in the one-dimensional case. Furthermore, 
we also refer to recent work on a similar problem in a related nonlocal reaction-diffusion 
system \cite{NadinRossiRyzhikPerthame}. \medskip 

\textbf{Acknowledgments:} CK would like to thank the Austrian Academy of
Sciences (\"{O}AW) for support via an APART fellowship and the European Commission (EC/REA) 
for support by a Marie-Curie International Re-Integration Grant. We also would like to thank 
two anonymous referees, whose very helpful comments led to improvements of this work. 

\appendix

\section{Convolutions in Weighted Spaces}
\label{ap:LiebLoss}

For convenience we record here a few facts about convolution operators on weighted spaces.
A standard result is the following:

\begin{thm}(Generalized Young's Inequality, \cite[p.13]{Folland1})
\label{thm:Folland}
Let $(S,\mu)$ be a sigma-finite measure space, and let $1\leq p\leq \I$ and $C>0$. Suppose $\kappa(x,y)$
is a measurable function on $S\times S$ such that
\beann
&\int_S |\kappa(x,y)|~\txtd\mu(y)\leq C\qquad \text{for all $x\in S$},\\
&\int_S |\kappa(x,y)|~\txtd\mu(x)\leq C\qquad \text{for all $y\in S$},
\eeann
and suppose $f\in L^p(S)$. Then the function
\benn
Tf(x):=\int_S \kappa(x,y) f(y) ~\txtd\mu(y)
\eenn
is well-defined almost everywhere and $Tf\in L^p(S)$. Furthermore, we have the estimate
\be
\|Tf\|_p\leq C\|f\|_p.
\ee
\end{thm}

We can apply the generalized Young's inequality to the space 
\benn
S=\R^d,\qquad \txtd\mu(x):=w(x)\txtd x,~w(x)=\frac{1}{(1+\|x\|^2)^l}
\eenn
for $\frac12<l<\I$ as discussed in Section \ref{sec:func_spaces}. Furthermore, suppose we only 
assume that $\phi \in L^1(\R^d)$ ({i.e.}~it may not be normalized to $\|\phi\|=1$) and set 
\benn
\kappa(x,y):=\phi_\sigma(x-y),\qquad \text{where }
\phi_\sigma(x)=\frac{1}{\sigma^d}\phi\left(\frac{x}{\sigma}\right).
\eenn
Then one is lead to the following direct calculation
\be
\label{eq:mini_calc}
\int_S \phi_\sigma (x-y) ~\txtd\mu(y)=\int_{\R^d} \phi_\sigma(x)w(x-y)~\txtd x=
\frac{1}{\sigma^d}\int_{\R^d} \frac{\phi(\tilde{x})}{(1+\|\sigma\tilde{x}-y\|^2)^l} \sigma^d ~\txtd\tilde{x}\leq \|\phi\|_1.
\ee

\begin{remark} 
Note that we crucially use here that 
our choice of weight is quite 'tame' in the sense that it is absolutely continuous with respect to
standard Lebesgue measure. Singular and/or growing weight functions would present a problem which
would have to be compensated by additional regularity assumptions on the kernel $\phi$.
\end{remark}

Hence, we may take $C=1=\|\phi\|_1$ in Theorem \ref{thm:Folland} and the previous discussion to obtain 
a weighted version of the more standard Young's inequality.

\begin{prop}(Young's inequality, see also \cite[p.14]{Folland1})
\label{prop:Young}
If $\phi\in L^1(\R^d)$ and $f\in L^p(\R^d;w_X)$ for $1\leq p< \I$ then $f\ast \phi_\sigma\in L^p(\R^d;w_X)$
and
\benn
\|f\ast \phi_\sigma\|_{p;w_X}\leq \|\phi \|_1\|f\|_{p;w_X}.
\eenn
\end{prop}

\begin{proof}
Note that $\phi\in L^1(\R^d)$ implies that $\phi_\sigma\in L^1(\R^d;w_X)$ for each $\sigma>0$. Thus,
applying Theorem \ref{thm:Folland} and the calculation \eqref{eq:mini_calc} yield the result.
\end{proof}

\begin{thm}(\cite[p.64]{LiebLoss})
\label{thm:LiebLoss}
Let $\phi\in L^1(\R^d)$ with $\|\phi\|=1$. Let $f\in L^p(\R^d,w_X)$ for some $p\in[1,\I)$ then
$\phi_\sigma \ast f\ra f$ in $L^p(\R^d;w_X)$ for $\sigma\ra 0$.
\end{thm}

\begin{proof}
The proof from \cite[p.65-66]{LiebLoss} applies verbatim as we have a weighted version of 
Young's inequality given in Proposition \ref{prop:Young} for our choice of weight function $w_X$ in \eqref{eq:weights}.
\end{proof}

It is clear that Proposition \ref{prop:Young} and Theorem \ref{thm:LiebLoss} would also
hold true for more general classes of weight functions which include as a subset the family $w_X$,
parametrized by the choice of $l$.

\section{Purely Oscillatory Solutions}
\label{ap:osc}

In this section we provide a proof of the result stated in Proposition \ref{prop:oscillatory}
which we re-state here for convenience:
 
\begin{prop}
\label{prop:osc_ap}
Suppose (B) holds with weight function $w_X\in L^1(\R^d)$. Suppose $\cL^\alpha_0U=0$ and 
$U$ is in the linear hull of simple exponential solutions. Then $U\in X$ if and only if 
$U$ is in the linear hull of purely oscillatory solutions.
\end{prop}

\begin{proof}
If $U$ is in the linear hull of purely oscillatory solutions it easily follows that $U\in X$
by using the global boundedness of a finite sum of sines and cosines together with the 
polynomial weight function $w_X$. Therefore, it remains to show the converse.

Suppose $U\in X$ and $\cL^\alpha_0U=0$. Then elliptic regularity \cite[Cor.~{11.4.13}]{Hoermander2} 
implies that $U\in C^\I(\R^d)$; in fact, $U$ is even real analytic. Proposition 
\ref{prop:Hoermander_sol} yields that $U$ can be expressed as a linear combination of 
exponential solutions 
\be
\label{eq:real_analytic}
U(x)=\sum_{k=1}^{K} a_{k}\txte^{\txti\xi_{k}^T x}
\ee
for some $a_k\in\C$, a natural number $K>0$, $\xi_{k}\in\C^d$ are distinct vectors and 
the series converges pointwise for each $x\in\R^d$. Note that the sum \eqref{eq:real_analytic} 
has a finite number of terms since we assumed that $U$ is in the linear hull and we do not have to 
take a closure. A solution is purely oscillatory if $\xi_{k}\in\R^d$ for all $k\in\N$. We argue by 
contradiction and suppose there exists at least one vector $\xi_{k}\not\in\R^d$. Define the following 
set of indices
\benn
\cK:=\{k\in\{1,2,\ldots,K\}:\xi_{k}\not\in\R^d\}
\eenn
and denote a maximum growth or decay exponent index pair by $(k^*,l^*)$, {i.e.}, 
$|\textnormal{Im}((\xi_{k^*})_{l^*})|\geq |\textnormal{Im}((\xi_{k})_l)|$ for all $k\in\cK$ and 
$l\in\{1,2,\ldots,d\}$. As a first case, suppose that the pair $(k^*,l^*)$ is unique.

Without loss of generality we are going to assume that $\textnormal{Im}(\xi_{k^*})_{l^*}>0$ 
since a coordinate transformation of the form $x_{l^*}\mapsto -x_{l^*}$ is going to cover the 
case when $\textnormal{Im}(\xi_{k^*})_{l^*}<0$. Furthermore, we may assume without loss of generality that 
$l^*=1$, respectively $k^*=1$, since a permutation of the coordinate indices, respectively the summands,
can always be applied. Let $(\xi_{1})_{1}=\mu-\lambda \txti$ with $\lambda>0$, $\mu\in\R$ and focus on 
a summand of $U(x)$ given by 
\benn
a_{1}\txte^{\txti\xi_{1}^T x}=a_{1}\txte^{\lambda x_1}\txte^{\mu \txti x_1}
\txte^{\txti\left((\xi_1)_2 x_2+\cdots +(\xi_1)_d x_d\right)}=:a_{1}\txte^{\lambda x_1}s(x).
\eenn
Observe that $\dim(\{x\in\R^d:s(x)=0\})\leq d-1$ as the zero set consists of the zeros of a product of 
trigonometric functions. Fix some $\delta>0$ and use the last observation and the fact that 
$\lim_{x_1 \ra \I} \txte^{\lambda x_1}/w_X(x)=+\I$ to conclude the existence 
of a sequence of disjoint balls $B(y_m,\delta)$ with centers $y_m\in\R^d$, $\|y_m\|\ra \I$, and radius 
$\delta$ such that 
\be
\label{eq:moving_balls}
\left|a_{1}\txte^{\lambda x_1}s(x)+\sum_{k=2}^{K} a_{k}\txte^{\txti\xi_{k}^T x}\right|^p w_X(x)>1
\quad \text{for $x\in B(y_m,\delta)$.} 
\ee
Note that the lower bound $1$ is just chosen for convenience and any other fixed positive constant would 
work as well. As a last step, we shall show that $U\not\in L^p(\R^d;w_X)$ for $1\leq p<\I$. We have
\beann
\|U\|_{L^p(\R^d;w_X)}^{p}&\geq&
\sum_{m=1}^\I\int_{B(y_m,\delta)}\left|\sum_{k=1}^{K} a_{k}\txte^{\txti\xi_{k}^T x}\right|^p w(x)~\txtd x\\
&\stackrel{\eqref{eq:moving_balls}}{\geq}&\sum_{m=1}^\I\int_{B(y_m,\delta)}~\txtd x 
=\sum_{m=1}^\I\frac{\pi^{d/2}\delta^d}{\Gamma\left(\frac{d}{2}+1\right)} =+\I.
\eeann
Therefore, it clearly follows that $U\not\in X$ as $X\subset L^p(\R^d,w_X)$. This concludes the case 
when the maximum growth or decay exponent index pair by $(k^*,l^*)$ is unique.

When $(k^*,l^*)$ is not unique, we have a finite number of index pairs which could be maximal 
and two cases may occur. First, if $k^*$ is unique then we may have dominating terms with exponentials
of the form $a_{k^*}\exp(-\lambda x_1+\lambda x_l)(\cdots)$ for some $l\in\{2,3,\ldots,d\}$. The 
construction of the balls in \eqref{eq:moving_balls} still works as the linear subspace 
$\{x\in\R^d:x_1=x_l\}$ has dimension $d-1$ {i.e.}~we can always arrange the balls $B(y_m,\delta)$ 
such that $\{x\in\R^d:x_1=x_l\}\cap B(y_m,\delta)=\{\}$. If $k^*$ is not unique, a similar 
argument applies. 

If the sum \eqref{eq:real_analytic} has terms of the form 
$a_k\txte^{\txti\xi_k^Tx}+a_j\txte^{\txti \xi_j^Tx}$, where $a_j=\overline{a_k}$ and 
$\xi_j=\overline{\xi_k}\not\in\R^d$ are complex conjugates (as the function $U$ has to be real-valued), 
then the first term leads to $a_k\txte^{\lambda x_l}(\cdots)$ while the second term 
contributes $\overline{a_k}\txte^{-\lambda x_l}(\cdots)$ for some coordinate $x_l$. Suppose 
without loss of generality that $\lambda>0$ then we may always choose $x_l$ sufficiently large, and
avoid the zeros of the multiplying trigonometric factors, to construct balls as above.

If there are terms of the form $\txte^{\lambda x_1+\txti ax_2+\cdots}-\txte^{\lambda x_1+\txti bx_2+\cdots}$ 
with $\lambda,a,b\in\R$ then there is the possibility of cancellations. However, since we know that
the vectors $\xi_k$ are distinct, we may assume that $a\neq b$. Hence, we have terms of the form
$\txte^{\lambda x_1}s_a(x)-\txte^{\lambda x_1}s_b(x)$, where the zero sets of the trigonometric functions 
$s_a$ and $s_b$ each have dimension $d-1$. In particular, the sets where $\txte^{\lambda x_1}s_a(x)$ and 
$-\txte^{\lambda x_1}s_b(x)$ are both positive have full dimension $d$ and repeat periodically due
to the periodicity of the trigonometric prefactors $s_a$ and $s_b$.

Indeed, the same arguments apply to all cases where the growth and decay indices are not unique, since 
$K$ is finite and the $\xi_k$ are distinct vectors, we can always construct the balls $B(y_m,\delta)$
to make the polynomially weighted $L^p$-norm of $U$ arbitrarily large with a finite number of balls. 
Hence we may conclude that $\xi_{k}\in\R^d$ for all $k$ if $U\in X$.
\end{proof}

The next step is to generalize Proposition \ref{prop:osc_ap} to the case of the closure of the linear
hull of exponential solutions. Although the 
next statement looks quite natural, it does not seem easy to prove so we leave it here as a conjecture 
for future work.

\begin{cjt}
\label{prop:osc_ap1}
Proposition \ref{prop:osc_ap} holds when the linear hull is replaced by its closure.
\end{cjt}

The problem in generalizing the proof to the closure of the linear hull is that one has to take 
care of potential intricate cancellation effects in the infinite series involved, as well as dealing 
with zeros of infinite sums of trigonometric functions.

\section{Invertibility - A Direct Proof}
\label{ap:direct}

Recall that we have used in Proposition \ref{prop:L1} a particular choice of weights 
$w_X$, $w_Y$ and exponent $p=2$ to use a known result about invertibility of the linear
operator 
\benn
\cL_1:X\ra Y, \qquad \cL_1U=\Delta U-U,
\eenn
where $X$ and $Y$ are chosen as 
\be
\label{eq:spaces_ap}
X=W^{k+2,p}(\R^d;w_X)\cap W^{k,2p}(\R^d;w_X)\qquad \text{and}\qquad Y=W^{k,p}(\R^d;w_Y)
\ee
for some $k\in\N_0$ and the norm on $X$ is given by 
\benn
\|u\|_X=\max\{\|u\|_{k+2,p;w_X},\|u\|_{k,2p;w_X}\}.
\eenn
Here we sketch the main arguments how a more direct proof of invertibility could be carried out. 
The basic assumption is that the positive weight functions $w_X, w_Y\in L^1(\R^d)$
are chosen such that $\cL_1$ is a well-defined linear operator and $1<p<\I$. If we restrict the 
class of weight function to exclude exponential growth then we obtain a trivial nullspace for $\cL_1$.

\begin{lem}
\label{lem:invertibleL1}
There exist weight functions $w_X,w_Y\in L^1(\R^d)$ such that the following statement holds: 
$\cL_1U=0$ for $U\in X$ if and only if $U\equiv0$.
\end{lem}
 
\begin{proof}(Sketch of an alternative to Proposition \ref{prop:L1})
We argue by contradiction and suppose we may find a nonzero solution $U$. The idea of the 
proof is similar to the proof of Proposition \ref{prop:Hoermander_sol}. By a result on homogeneous 
linear PDE \cite[Thm 10.5.1,p.39]{Hoermander2} we may write any solution of $\Delta U-U=0$ in $X$, 
for any choice of $w\in L^1(\R^d)$, by a limit of basic exponential solutions, 
\benn
U(x)=\sum_{j=1}^\I f_{k}(x)\txte^{\txti\xi_{k}^T x},
\eenn
where the sequences $f_{k}(x)$ are polynomials and $\xi_{k}\in\C^d$ are vectors. If there exists a 
component of $\xi_{k}$ with nonzero imaginary part then $U(x)$ does have a component behaving like 
$\txte^{-x}$ or $\txte^{x}$ and we have seen in Section \ref{sec:implicit}, as well as Appendix \ref{ap:osc}, 
that we may choose a weight function with polynomial decay as $\|x\|\ra \I$ so that $U\not\in X$. 
Hence, it follows that $\xi_{k}\in\R^d$. If $f_{k}(x)$ is a non-constant polynomial for some $k$, then we may 
choose a positive weight function $w\in L^1(\R^d)$ such that $|f_{k}(x)|^pw(x)\ra+\I$ as $\|x\|\ra \I$. In 
this case, if follows that $U\not\in X$. Therefore, 
$f_{k}(x)=a_{k}$ for some constants $a_{k}$ and the possible solutions are all purely oscillatory
or homogeneous. Obviously we may assume for each fixed $k$ that $\xi_{k_1}\neq\xi_{k_2}$ for 
$k_1\neq k_2$. Applying the differential operator to $U$ yields
\be
\label{eq:conv_sol1}
0=(\cL_1U)(x)=\sum_{j=1}^\I a_{k} (-\|\xi_{k}\|^2-1) \txte^{\txti\xi_{k}^T x}.
\ee
Since $\xi_{k_1}\neq\xi_{k_2}$ for $k_1\neq k_2$ it follows from \eqref{eq:conv_sol1} that 
$a_k(\|\xi_{k}\|^2+1)=0$ for all $k\in\N$. Thus, we get that that $a_k=0$ for all $k$. 
Therefore, it follows that $U\equiv 0$ providing the required contradiction.
\end{proof}

It is important to note that the last proof does not employ the particular choice (B) given in 
\eqref{eq:weights} for the weight functions and the Sobolev space exponent. Weight
functions with rather general polynomial growth and $p\in(1,\I)$ suffice not only to 
prove Lemma \ref{lem:invertibleL1} but also to establish continuity and differentiability
properties of $F$ as long as \eqref{eq:weights_ineq1} holds; see also Section \ref{sec:func_spaces}.
Based on the knowledge of the nullspace we may now proceed to investigate invertibility.

\begin{lem}
\label{lem:invertibleL1_final}
There exists a weight function $w\in L^1(\R^d)$ such that $\cL_1:X\ra Y$ is invertible.
\end{lem}

\begin{proof}(Sketch of an alternative to Proposition \ref{prop:L1})
For the proof we shall need as auxiliary spaces the 
H{\"{o}}lder spaces $C^{k+\gamma}(\R^d)$ for $k\in\N_0$ and $\gamma\in(0,1)$ as defined in Section 
\ref{sec:func_spaces}. By Lemma \ref{lem:invertibleL1} we know that $\cL_1U=0$ for $U\in X$ 
implies that $U\equiv 0$. Since
\beann
\|u\|_{X}&=&\max\{\|u\|_{k+2,p;w},\|u\|_{k,2p;w}\}\leq \|u\|_{k+2,p;w}+\|u\|_{k,2p;w}\\
&\leq&  2\|u\|_{C^{k+2+\gamma}(\R^d)} ~\max_{\rho\in\{p,2p\}}\|w\|_{L^1(\R^d)}^{1/\rho}
\eeann
it follows that $C^{k+2+\gamma}(\R^d)\subset X$; see also \cite{MitinaTyurin}. Hence 
$\cL_1U=0$ for $U\in C^{k+2+\gamma}(\R^d)$ 
implies that $U\equiv 0$ so that $\cL_1$ has trivial nullspace as an operator 
\benn
\cL_1:C^{k+2+\gamma}(\R^d)\ra C^{k+\gamma}(\R^d). 
\eenn
Applying Mukhamadiev's Theorem \cite{Mukhamadiev} it follows that 
$\cL_1:C^{k+2+\gamma}(\R^d)\ra C^{k+\gamma}(\R^d)$ is invertible. To show invertibility on 
$X$ an approximation argument can be used similar to the strategy in \cite{MitinaTyurin}.
Let $\kappa\in C_c^\I(\R^d)$ with $\|\kappa\|_{L^1(\R^d)}=1$ and consider 
$\kappa_\epsilon(x)=\epsilon^{-d}\kappa(x/\epsilon)$ so that 
$\|\kappa_\epsilon\|_{L^1(\R^d)}=\|\kappa\|_{L^1(\R^d)}$. Let $f\in L^p(\R^d;w)$ and 
define $f_\epsilon:=\kappa_\epsilon\ast f$. Then a standard analysis result (\cite[p.64]{LiebLoss}, 
see also Appendix \ref{ap:LiebLoss}) implies that 
\benn
f_\epsilon \in C^\I(\R^d) \qquad \text{and} \qquad f_\epsilon\ra f ~\text{in $L^p(\R^d;w)$.}
\eenn
Due to the invertibility of $\cL_1:C^{k+2+\gamma}(\R^d)\ra C^{k+\gamma}(\R^d)$ it follows that 
the equation
\benn
\cL_1u =f_\epsilon
\eenn
has a unique solution $u_\epsilon\in C^{k+2+\gamma}(\R^d)$ for each $\epsilon>0$. The desired result 
will follow if one can show that $u_\epsilon$ converges to some function $u_0\in X$ as $\epsilon \ra 0$
but the convergence essentially follows from a suitable coercivity estimate for the operator 
$\cL_1$ as discussed in \cite{MitinaTyurin}.
\end{proof}

The main point of the last proof is that we have picked a more complicated weak solution space $X$
but the invertibility of the restricted operator onto H\"{o}lder space, in combination with 
an approximation argument, can be used to yield invertibility for $\cL_1:X\ra Y$. Note that one still has to make
some restrictions on $w_X$, $w_Y$, $p$ but that no explicit choice is necessary if a direct investigation
of the linearized operator is considered.


\begin{thebibliography}{10}

\bibitem{AlfaroCoville}
M.~Alfaro and J.~Coville.
\newblock Rapid travelling waves in the nonlocal {Fisher} equation connect two
  unstable states.
\newblock {\em Appl. Math. Lett.}, 25(12):2095--2099, 2012.

\bibitem{AlfaroCovilleRaoul}
M.~Alfaro, J.~Coville, and G.~Raoul.
\newblock Bistable travelling waves for nonlocal reaction diffusion equations.
\newblock {\em arXiv:1303.3554v1}, pages 1--16, 2013.

\bibitem{AshwinBartuccelliBridgesGourley}
P.~Ashwin, M.V. Bartuccelli, T.J. Bridges, and S.A. Gourley.
\newblock Travelling fronts for the {KPP} equation with spatio-temporal delay.
\newblock {\em Zeitschr. Angewand. Math. Phys.}, 53(1):103--122, 2002.

\bibitem{BerestyckiHamelNadirashvili}
H.~Berestycki, F.~Hamel, and N.~Nadirashvili.
\newblock The speed of propagation for {KPP} type problems. {I. Periodic}
  framework.
\newblock {\em J. Eur. Math. Soc.}, 7(2):173--213, 2005.

\bibitem{BerestyckiHamelNadirashvili1}
H.~Berestycki, F.~Hamel, and N.~Nadirashvili.
\newblock The speed of propagation for {KPP} type problems. {II. General}
  domains.
\newblock {\em J. Amer. Math. Soc.}, 23:1--32, 2010.

\bibitem{BerestyckiHamelRoques}
H.~Berestycki, F.~Hamel, and L.~Roques.
\newblock Analysis of the periodically fragmented environment model: {I -
  Species} persistence.
\newblock {\em J. Math. Biol.}, 51(1):75--113, 2005.

\bibitem{BerestyckiNadinPerthameRyzhik}
H.~Berestycki, G.~Nadin, B.~Perthame, and L.~Ryzhik.
\newblock The non-local {Fisher-KPP} equation: travelling waves and steady
  states.
\newblock {\em Nonlinearity}, 22:2813--2844, 2009.

\bibitem{Britton1}
N.F. Britton.
\newblock Aggregation and the competitive exclusion principle.
\newblock {\em J. Theoret. Biol.}, 136:57--66, 1989.

\bibitem{Britton}
N.F. Britton.
\newblock Spatial structures and periodic travelling waves in an
  integro-differential reaction-diffusion population model.
\newblock {\em SIAM J. Appl. Math.}, 50(6):1663--1688, 1990.

\bibitem{CovilleDupaigne1}
J.~Coville and L.~Dupaigne.
\newblock Propagation speed of travelling fronts in non local
  reactionвЂ“diffusion equations.
\newblock {\em Nonl. Anal. Theor. Meth. Appl.}, 60(5):797--819, 2005.

\bibitem{CovilleDupaigne}
J.~Coville and L.~Dupaigne.
\newblock On a non-local equation arising in population dynamics.
\newblock {\em Proc. R. Soc. Edinburgh A}, 137(4):727--756, 2007.

\bibitem{Deimling}
K.~Deimling.
\newblock {\em Nonlinear Functional Analysis}.
\newblock Dover, Mineola, NY, 2010.

\bibitem{del-Castillo-NegreteCarrerasLynch}
D.~del Castillo-Negrete, B.A. Carreras, and V.E. Lynch.
\newblock Front dynamics in reaction-diffusion systems with {Levy} flights: a
  fractional diffusion approach.
\newblock {\em Phys. Rev. Lett.}, 91(1):018302, 2003.

\bibitem{DuMa}
Y.~Du and L.~Ma.
\newblock Logistic type equations on {$\mathbb{R}^N$} by a squeezing method
  involving boundary blow-up solutions.
\newblock {\em J. London Math. Soc.}, 64(2):107--121, 2001.

\bibitem{DudleyNorvaisa}
R.M. Dudley and R.~Norvai{\v{s}}a.
\newblock {\em Concrete Functional Calculus}.
\newblock Springer, 2011.

\bibitem{EbertvanSaarloos}
U.~Ebert and W.~van Saarloos.
\newblock Front propagation into unstable states: universal algebraic
  convergence towards uniformly translating pulled fronts.
\newblock {\em Physica D}, 146:1--99, 2000.

\bibitem{FangZhao}
J.~Fang and X.Q. Zhao.
\newblock Monotone wavefronts of the nonlocal {Fisher-KPP} equation.
\newblock {\em Nonlinearity}, 24(11):3043--3054, 2011.

\bibitem{FelmerYangari}
P.~Felmer and M.~Yangari.
\newblock Fast propagation for fractional {KPP} equations with slowly decaying
  initial conditions.
\newblock {\em SIAM J. Math. Anal.}, 45(2):662--678, 2013.

\bibitem{FengHuangYang}
J.-J. Feng, L.~Huang, and S.-J. Yang.
\newblock Solutions of {Laplace} equation in {$n$}-dimensional space.
\newblock {\em Commun. Theor. Phys.}, 56(4):623--625, 2011.

\bibitem{Fisher}
R.A. Fisher.
\newblock The wave of advance of advantageous genes.
\newblock {\em Ann. Eugenics}, 7:353--369, 1937.

\bibitem{Folland1}
G.~Folland.
\newblock {\em Introduction to Partial Differenial Equations}.
\newblock Princeton University Press, 1976.

\bibitem{FurterGrinfeld}
J.~Furter and M.~Grinfeld.
\newblock Local vs. non-local interactions in population dynamics.
\newblock {\em J. Math. Biol.}, 27:65--80, 1989.

\bibitem{GenieysVolpertAuger}
S.~Genieys, V.~Volpert, and P.~Auger.
\newblock Pattern and waves for a model in population dynamics with nonlocal
  consumption of resources.
\newblock {\em Math. Model. Nat. Phenom.}, 1(1):63--80, 2006.

\bibitem{GindikinVolevich}
S.G. Gindikin and L.R. Volevich.
\newblock {\em Distributions and Convolutions Equations}.
\newblock Gordon and Breach Science Publishers, 1990.

\bibitem{Gourley}
S.A. Gourley.
\newblock Travelling front solutions of a nonlocal {Fisher} equation.
\newblock {\em J. Math. Biol.}, 41(3):272--284, 2000.

\bibitem{HamelRyzhik}
F.~Hamel and L.~Ryzhik.
\newblock On the nonlocal {Fisher-KPP equation}: steady states, spreading speed
  and global bounds.
\newblock {\em arXiv:1307.3001}, pages 1--47, 2013.

\bibitem{HarringtonRaich}
P.S. Harrington and A.~Raich.
\newblock Sobolev spaces and elliptic theory on unbounded domains.
\newblock {\em Adv, Differential Equations}, 19(7):635--692, 2014.

\bibitem{Hoermander2}
L.~H{\"{o}}rmander.
\newblock {\em The Analysis of Linear Partial Differential Operators {II}}.
\newblock Springer, 1983.

\bibitem{Hoermander1}
L.~H{\"{o}}rmander.
\newblock {\em The Analysis of Linear Partial Differential Operators {I}}.
\newblock Springer, 1990.

\bibitem{KolmogorovPetrovskiiPiscounov}
A.~Kolmogorov, I.~Petrovskii, and N.~Piscounov.
\newblock A study of the diffusion equation with increase in the amount of
  substance, and its application to a biological problem.
\newblock In V.M. Tikhomirov, editor, {\em {Selected Works of A. N. Kolmogorov
  I}}, pages 248--270. Kluwer, 1991.
\newblock {Translated by V. M. Volosov from Bull. Moscow Univ., Math. Mech. 1,
  1--25, 1937}.

\bibitem{KongShen}
L.~Kong and W.~Shen.
\newblock Positive stationary solutions and spreading speeds of {KPP} equations
  in locally spatially inhomogeneous media.
\newblock {\em Meth. Applic. Appl. Anal.}, 18(4):427--456, 2011.

\bibitem{Kufner}
A.~Kufner.
\newblock {\em Weighted Sobolev Spaces}.
\newblock Teubner, 1980.

\bibitem{Lee}
J.M. Lee.
\newblock {\em Introduction to Smooth Manifolds}.
\newblock Springer, 2006.

\bibitem{LefeverLejeune}
R.~Lefever and O.~Dejeune.
\newblock On the origin of tiger bush.
\newblock {\em Bull. Math. Biol.}, 59(2):263--294, 1997.

\bibitem{LiebLoss}
E.H. Lieb and M.~Loss.
\newblock {\em Analysis}.
\newblock AMS, 2nd edition, 2001.

\bibitem{MancinelliVergniVulpiani}
R.~Mancinelli, D.~Vergni, and A.~Vulpiani.
\newblock Front propagation in reactive systems with anomalous diffusion.
\newblock {\em Physica D}, 185(3):175--195, 2003.

\bibitem{MitinaTyurin}
O.A. Mitina and V.M. Tyurin.
\newblock On the invertibility of linear partial differential operators in
  {H{\"{o}}lder and Sobolev} spaces.
\newblock {\em Sbornik Math.}, 194(5):733--744, 2003.

\bibitem{Mukhamadiev}
E.M. Mukhamadiev.
\newblock On invertibility of elliptic partial differential operators.
\newblock {\em Dokl. Akad. Nauk SSSR}, 205:1292--1295, 1972.
\newblock English transl. in Soviet Math. Dokl. 13, 1972.

\bibitem{Murray1}
J.D. Murray.
\newblock {\em Mathematical Biology I: An Introduction}.
\newblock Springer, 3rd edition, 2002.

\bibitem{NadinPerthameTang}
G.~Nadin, B.~Perthame, and M.~Tang.
\newblock Can a traveling wave connect two unstable states? {The} case of the
  nonlocal {Fisher} equation.
\newblock {\em Comptes Rendus Math.}, 349(9):553--557, 2011.

\bibitem{NadinRossiRyzhikPerthame}
G.~Nadin, L.~Rossi, L.~Ryzhik, and B.~Perthame.
\newblock Wave-like solutions for nonlocal reaction-diffusion equations: a toy
  model.
\newblock {\em Math. Model. Nat. Phen.}, 8(3):33--41, 2013.

\bibitem{Okubo}
A.~Okubo.
\newblock {\em Diffusion and Ecological Problems}.
\newblock Springer, 1980.

\bibitem{PerthameGenieys}
B.~Perthame and B.~G\'{e}nieys.
\newblock Concentration in the nonlocal {Fisher} equation: the
  {Hamilton-Jacobi} limit.
\newblock {\em Math. Mod. Nat. Phenom.}, 2(4):135--151, 2007.

\bibitem{PerthameSouganidis}
B.~Perthame and P.E. Souganidis.
\newblock Front propagation for a jump process arising in spatial ecology.
\newblock {\em Discrete Contin. Dyn. Syst.}, 13(5):1235--1246, 2005.

\bibitem{PucciSerrin}
P.~Pucci and J.~Serrin.
\newblock {\em The Maximum Principle}.
\newblock Birkh{\"{a}}user, 2007.

\bibitem{Rossi}
L.~Rossi.
\newblock Non-existence of positive solutions of fully nonlinear elliptic
  equations in unbounded domains.
\newblock {\em Comm. Pure Appl. Anal.}, 7(1):125--141, 2008.

\bibitem{VougalterVolpert}
V.~Vougalter and V.~Volpert.
\newblock On the existence of stationary solutions for some {non-Fredholm}
  integro-differential equations.
\newblock {\em Doc. Math.}, 16:561--580, 2011.

\bibitem{WangLiRuan}
Z.C. Wang, W.T. Li, and S.~Ruan.
\newblock Existence and stability of traveling wave fronts in reaction
  advection diffusion equations with nonlocal delay.
\newblock {\em J. Differential Equat.}, 238(1):153--200, 2007.

\bibitem{Zou}
X.~Zou.
\newblock Delay induced traveling wave fronts in reaction diffusion equations
  of {KPP-Fisher} type.
\newblock {\em J. Comp. Appl. Math.}, 146(2):309--321, 2002.

\end{thebibliography}
\end{document}